\documentclass[letterpaper, 10pt, onecolumn]{cssconf}
\IEEEoverridecommandlockouts
\usepackage{cite}
\usepackage{amsmath, amssymb, amsfonts, mathtools}
\usepackage{algorithm}
\usepackage{algorithmic}
\usepackage{float}
\usepackage{graphicx}
\usepackage{textcomp}
\usepackage{xcolor}
\usepackage{pgfplots}
\usepgfplotslibrary{groupplots}
\usetikzlibrary{calc}
\pgfplotsset{compat=1.18}
\makeatletter
\let\NAT@parse\undefined
\makeatother
\usepackage[
    colorlinks=false,
]{hyperref}

\usepackage{booktabs}
\usepackage{siunitx}
\usepackage{multirow}

\newtheorem{theorem}{Theorem}

\newtheorem{definition}{Definition}

\newtheorem{proposition}{Proposition}
\newtheorem{assumption}{Assumption}

\newtheorem{remark}{Remark}

\usepackage{mathtools}

\newcommand{\R}{\mathbb{R}}

\newcommand{\N}{\mathbb{N}}

\newcommand{\vct}[1]{\pmb{#1}}
\newcommand{\mat}[1]{\mathbf{#1}}

\newcommand{\T}{^\top}

\newcommand{\psd}{\succeq0}

\newcommand{\brackets}[1]{\left[ #1 \right]}
\newcommand{\paren}[1]{\left(#1 \right)}

\newcommand{\vecc}[1]{\begin{bmatrix} #1 \end{bmatrix}}

\newcommand{\obar}[1]{\overline{#1}}
\newcommand{\ubar}[1]{\underline{#1}}

\DeclarePairedDelimiter{\abs}{\lvert}{\rvert}

\DeclareMathOperator{\diag}{diag}
\DeclareMathOperator{\trace}{tr}

\DeclareMathOperator{\E}{\mathbb{E}}

\begin{document}

\title{Beyond Ellipsoids: Semialgebraic Tightening for Chance Constraints Under Actuator Saturation}

\author{
    Carlo Karam, Mirko Fiacchini, and Matteo Tacchi-Bénard
\thanks{
    The authors are with Univ. Grenoble Alpes, Grenoble INP, CNRS, GIPSA-lab 
    (e-mail: \{carlo.karam, mirko.fiacchini, matteo.tacchi\}@gipsa-lab.fr).
}
\thanks{Matteo Tacchi-Bénard's work is supported by the French National Research Agency (ANR) under the PIVOINE project grant
ANR-25-CE48\-1599\-01.}
}

\maketitle

\begin{abstract}
    Motivated by stochastic model predictive control applications, we develop a semialgebraic framework for constraint tightening in chance-constrained systems with unbounded additive disturbances and saturated inputs. The saturation nonlinearity is represented exactly through its semialgebraic description. A polynomial Lyapunov function satisfying a stochastic drift condition is then synthesized using sum-of-squares optimization, yielding finite-time probabilistic reachable sets and probabilistic ultimate bound. The resulting sets are further optimized through an objective tailored to constraint tightening. Numerical examples illustrate the effectiveness and scalability properties of the approach.
\end{abstract}


\section{Introduction}\label{sec:introduction}
Reachability analysis provides a formal framework for characterizing the propagation of dynamical systems trajectories given a set of initial conditions, admissible inputs, and disturbances. This is crucial in safety-critical control, where post-hoc validation probes only finitely many realizations and therefore cannot exclude rare but possible constraint violations. Specifically, in predictive control, reachability analysis becomes a synthesis tool: reachable error sets are often computed offline to tighten state constraints before solving the online optimization, ensuring robust performance under disturbances.

For dynamical systems with bounded uncertainty, established reachability analysis methods range from ellipsoidal and zonotopic approximations to polytopic and support function techniques~\cite{kurzhanskiy2007ellipsoidal, girard2005reachability, leguernic2009reachability, blanchini2015set, althoff2021set}. More recently, sum-of-squares (SOS) programming has enabled rigorous semialgebraic approximations of forward reachable sets for polynomial systems~\cite{lasserre2001global, magron2019discrete,korda2021spatio}. These robust methods provide worst-case guarantees but become overly conservative, or yield unbounded reachable sets, when disturbances are stochastic with unbounded support, leading to potentially empty tightened sets.

Stochastic reachability addresses this limitation by replacing absolute containment with probabilistic bounds at a prescribed violation level. Dynamic programming has been used to compute finite-horizon safety probabilities and probabilistic safe sets for controlled stochastic hybrid systems~\cite{abate2008probabilistic}. Functional approaches instead search for stochastic barrier or Lyapunov functions, often enforced through SOS relaxations~\cite{prajna2007framework}. Related certificate-based methods relax one-step supermartingale requirements through $k$-inductive stochastic barriers~\cite{anand2022kinductive}. Other recent barrier-like approaches address finite and infinite-horizon safety and reach-avoid verification of stochastic discrete-time systems, with polynomial certificates computed through SOS programming~\cite{Xue2025finite,Xue2026barrier}. More notably, SOS certificates have recently been extended to guarantee almost-sure reachability for stochastic polynomial systems~\cite{kordabad2025almost}.

In Stochastic Model Predictive Control (SMPC), probabilistic reachability enters through stochastic tube constructions, which decompose the state into a nominal component  (optimized via nominal MPC) and a stochastic error component, which is stabilized and bounded offline by a probabilistic reachable set (PRS)~\cite{mesbah2016stochastic, cannon2011stochastic, fiacchini2021probabilistic, hewing2018stochastic,fiacchini2026measured}. State constraint-tightening by this PRS ensures that the original state chance constraints are satisfied. However, most tube-based or constraint-tightening SMPC constructions rely on an ideal linear feedback law to stabilize the error, which can become arbitrarily large under unbounded disturbances, surpassing actuator saturation limits. The resulting nonlinearity in the error dynamics can invalidate the derived PRS. Existing stochastic tube-based saturated SMPC approaches recover guarantees through convex embeddings~\cite{KaramPRS2025, karam2026SMPC}, but the resulting ellipsoidal bounds remain conservative, in large part due to their inherent symmetry.

The proposed framework moves beyond these ellipsoidal bounds by exploiting an exact semialgebraic representation of the saturation nonlinearity~\cite{korda2022stability}. Specifically, the set of admissible saturated inputs is described by polynomial inequalities, allowing a stochastic drift condition to be imposed directly over the resulting semialgebraic set. The probabilistic reachability problem is then cast as a search for SOS-based Foster-Lyapunov functions satisfying this drift criterion, whose sublevel sets serve as guaranteed over-approximations of the true transient PRS and steady-state probabilistic ultimate bound (PUB) of the stochastic error process. Due to the inherent scalability limitations of SOS programming, this method is primarily intended for moderately sized systems. 

\subsection{Notation}
Let $\R$ and $\N$ denote the real and natural numbers, respectively. Let $\vct{1}_n \in \mathbb{R}^n$ denote the vector of ones and $\mat{I}_n$ the identity matrix of dimension $n$. For $x \in \R^n$, $\|x\|_\infty$ is the infinity norm, $x_i$ denotes the $i$-th entry, and $|x|$ is taken componentwise. The operators $\diag(\cdot)$ and $\trace(\cdot)$ denote the diagonal matrix constructor and the trace operator. The elementwise, or Hadamard, product is denoted by $\odot$. Probability and expectation are denoted by $\Pr\{\cdot\}$ and $\mathbb{E}[\cdot]$, respectively. $\mathbb{R}[y]$ is the polynomial ring in $y$, and $\Sigma[y] \subset \mathbb{R}[y]$ is the sum-of-squares (SOS) cone. Vectors of monomials in $y$ up to degree $r$ are denoted by $\Phi_{0:r}(y)$, while $\Phi_{1:r}(y)$ strictly excludes the constant (degree $0$) term. For sets $\mathcal{A}, \mathcal{B} \subseteq \mathbb{R}^n$, the Pontryagin difference is $\mathcal{A} \ominus \mathcal{B} = \left\{a \in \R^n \mid a + b \in \mathcal{A},\ \forall b\in\mathcal{B} \right\}$, and the support function in direction $q$ is $\delta_{\mathcal{A}}(q) = \sup_{a\in\mathcal{A}} q^\top a$.

\section{Preliminaries and Problem Definition}\label{sec:prelims}
This section formally introduces the saturated stochastic system under consideration, defines the probabilistic reachability concepts required for constraint tightening, and details an exact semialgebraic representation of the saturated error dynamics.
\subsection{Problem Statement}\label{subsec:prob-statement}
Consider the uncertain discrete-time system 
\begin{equation}
    \label{eq:base-lti}
    x_{k+1} = Ax_k + B u_k + w_k
\end{equation}
where $x_k, w_k \in \R^n$, $u_k \in \R^m$, and the matrices $A \in \R^{n \times n}, B \in \R^{n \times m}$ are known, with $A$ Schur-stable. 
\begin{assumption}
The process $w_k \sim \mathbb{P}_w$ is i.i.d.\ and satisfies $\E[w_k] = 0$. Moreover, the probability measure $\mathbb{P}_w$  admits finite moments up to the required degree.
\end{assumption}
The admissible state and input sets are
\begin{equation*}
\mathcal{X} = \{x \in \R^n \mid H x \leq h \}, \quad \mathcal{U} = \{ u \in \R^m \mid \|u\|_\infty \leq 1 \},
\end{equation*}
where $H \in \mathbb{R}^{n_c \times n}$ and $h \in \mathbb{R}^{n_c}$, with $n_c$ denoting the number of facets. Note that $\mathcal{U}$ is restricted to the unit hypercube to simplify presentation. In SMPC, the state chance constraints
\begin{equation}
    \label{eq:state-chance-constraints}
    \Pr \{ x_k \in \mathcal{X} \mid x_0 \} \geq 1 - \varepsilon, \quad \forall k \in \N,    
\end{equation}
must be enforced for a desired violation probability $\varepsilon \in (0, 1)$, together with the hard input constraint $u_k \in \mathcal{U}$, for all $k \in \N$. Direct treatment of these chance constraints is generally computationally prohibitive. A standard approach therefore relies on constraint tightening based on the decomposition 
\begin{equation}
x_k = z_k + e_k
\end{equation}
where $z_k$ denotes the nominal state and $e_k$ is the disturbance-induced prediction error. Taking $z_0 = x_0$ yields $e_0 = 0$. The nominal dynamics are
\begin{equation}
z_{k+1} = Az_k + Bv_k
\end{equation}
where $v_k$ is the nominal control input generated by the MPC and satisfies $
    v \in \mathcal{V} = \{v \in \R^m \mid \underline{v} \leq v \leq \overline{v}\} \subseteq \mathcal{U}.
$
The possible strict inclusion $\mathcal{V} \subset \mathcal{U}$ allows part of the actuator range to be reserved for feedback action. A fixed linear feedback gain $K$ is applied to regulate the error. Since the disturbance has unbounded support, $v_k + K e_k$ cannot in general be guaranteed to remain in $\mathcal{U}$. The applied control is therefore $u_k = \varphi(v_k + K e_k)$, where $\varphi(\cdot)$ denotes componentwise saturation,
\begin{equation}\label{eq:sat}
    \varphi_i(y) =
        \begin{cases}
            1, \quad & y_i \geq 1 \\
            y_i,     & -1 \leq y_i \leq 1 \\
            -1,      & y_i \leq -1
        \end{cases}
        \quad \forall i = 1, \ldots, m,
\end{equation} 
ensuring $u_k \in \mathcal{U}$ by construction. This gives error dynamics of the form 
  \begin{equation}\label{eq:sat-error-dynamics}
    e_{k+1} = Ae_k+B\paren{\varphi(Ke_k+v_k)-v_k}+w_k. 
  \end{equation}

Chance constraints are then enforced indirectly by bounding, with high probability, the error within probabilistic sets. This concept is introduced below.
\begin{definition}[Probabilistic Reachable Set]
    The process $e_k$ with $e_0 = 0$ admits the set $\mathcal{R}_k^\varepsilon$ as a $k$-step probabilistic reachable set (PRS) with violation level
    $\varepsilon \in (0,1)$ if
    \begin{equation*}
        \Pr\{e_k\in\mathcal{R}_k^\varepsilon\}\geq1-\varepsilon .
    \end{equation*}
\end{definition}
\vspace*{0.2cm}
\begin{definition}[Probabilistic Ultimate Bound]
    The stochastic process $e_k$ admits the set $\mathcal{R}^\varepsilon$ as a probabilistic ultimate bound (PUB) with violation level
    $\varepsilon \in (0,1)$ if, for every initial condition $e_0\in\R^n$, there exists a finite time $t(e_0)$ such that
    \begin{equation*}
        \Pr\{e_k\in\mathcal{R}^\varepsilon\}\geq1-\varepsilon,
        \qquad \forall k\geq t(e_0).
    \end{equation*}
\end{definition}
\vspace*{0.2cm}
A PRS thus provides finite-time probabilistic containment, whereas a PUB bounds the steady-state error. Accordingly, at each time instant $k$, the nominal MPC enforces $z_k \in \mathcal{Z}_k = \mathcal{X}\ominus\mathcal{R}_k^\varepsilon$, and an analogous terminal tightening is obtained from the PUB.\@ If the error lies within the PRS or PUB at the desired violation level $\varepsilon$, then the constraint $z_k \in \mathcal{Z}_k$ guarantees that the original state chance constraints~\eqref{eq:state-chance-constraints} are satisfied.

Consequently, the construction of PRS and PUB sets constitutes the central offline task. These sets should capture, to the greatest feasible extent, the true propagation of the error dynamics induced jointly by input saturation and stochastic disturbances. Indeed, overly conservative PRS and PUB shrink the nominal admissible state set $\mathcal{Z}_k$, thereby reducing the feasible region of the SMPC and degrading closed-loop performance, potentially leading to infeasibility.

\subsection{Semialgebraic Representation of the Saturation}
Saturation is a hard nonlinearity whose exact treatment is too burdensome for most analysis problems. To obtain tractable conditions for probabilistic reachability, an appropriate representation or bound of the saturation term is essential. Most commonly, the saturation nonlinearity is embedded in some linear or convex inclusion, for instance through sector bounds or polytopic inclusion models~\cite{tarbouriechStability2011}. Because downstream sources of conservatism may already affect the quality of the obtained sets, this work prioritizes exact representations of the saturation nonlinearity. One such representation is the so-called regions of saturation model, which exploits the piecewise-affine (PWA) structure of $\varphi(\cdot)$. Its development is fairly standard and is detailed in Appendix~\ref{subsec:sat-regions}. While the PWA formulation is exact, it requires the enumeration of $3^m$ regions of saturation, which would only compound the inherent scalability issues of SOS optimization frameworks. As such, this work adopts an exact semialgebraic description of the saturation function, which is more appropriate for the proposed polynomial optimization approach. To this end, the semialgebraic representation of the saturation mentioned in~\cite{korda2022stability} is formalized and adapted for the SMPC reachability problem.

\begin{proposition}\label{prop:semialg-sat}
Let $\varphi: \R^m \to \R^m$ denote the componentwise saturation described in~\eqref{eq:sat}, and recall $\mathcal{V} = [\ubar{v}, \obar{v}]$. For any fixed $K \in \R^{m \times n}$, define
\begin{equation}\label{eq:sat-graph}
    \mathcal{G}_K = \left\{ (e, v, u) \in \R^{n+2m} \;\middle|\; v \in \mathcal{V}, u = \varphi(K e + v)\right\}.  
\end{equation}
Then, the set in~\eqref{eq:sat-graph} admits the equivalent semialgebraic representation
\begin{equation}\label{eq:sat-semialg}
    \mathcal{G}_K =
        \left\{
                (e, v, u) \;\middle|\;
                \begin{array}{l}
                    (v - \ubar{v}) \odot (\obar{v} - v) \geq 0,\\
                    (u + \pmb{1}_m) \odot (\pmb{1}_m - u) \geq 0,\\
                    (Ke + v - u) \odot (u + \pmb{1}_m)\ge0,\\
                    (u - Ke - v) \odot (\pmb{1}_m - u)\ge0
                \end{array}
        \right\}.
\end{equation}
\end{proposition}
\vspace{0.5em}
\begin{proof}
Since the saturation acts componentwise, it suffices to establish the result for $m = 1$. Let $q = Ke + v$ and define
\begin{equation*}
    \mathcal{S}
    =
    \left\{
        (q,u) \in \R^2
        \,\middle|\,
        -1 \leq u \leq 1,\,
        g^+(q,u) \geq 0,\,
        g^-(q,u) \geq 0
    \right\},
\end{equation*}
where
\begin{equation*}
    g^+(q,u) = (u - q)(-1 - u),
    \qquad
    g^-(q,u) = (u-q)(1-u).
\end{equation*}

We first show that $(q,u) \in \mathcal{S}$ implies $u = \varphi(q)$. Since $-1 \leq u \leq 1$, three cases are possible. If $u=-1$, then
\begin{equation*}
    g^+(q,u)=0,
    \qquad
    g^-(q,u)=2(-1-q)\geq0,
\end{equation*}
which implies $q\leq-1$, and therefore $u=\varphi(q)=-1$. 

Similarly, if $u=1$, then
\begin{equation*}
    g^-(q,u)=0,
    \qquad
    g^+(q,u)=2(q-1)\geq0,
\end{equation*}
so $q\geq1$ and $u=\varphi(q)=1$.

It remains to consider $-1<u<1$. In this case, $u+1>0$ and $1-u>0$, and hence
\begin{equation*}
    g^+(q,u)\geq0 \;\Longrightarrow\; q-u\geq0,
    \quad
    g^-(q,u)\geq0 \;\Longrightarrow\; u-q\geq0.
\end{equation*}
Thus $q=u$, which again gives $u=\varphi(q)$. Therefore,
\begin{equation*}
    (q,u)\in\mathcal{S}
    \;\Longrightarrow\;
    u=\varphi(q).
\end{equation*}
For the converse, suppose $u=\varphi(q)$, which by definition implies $-1\leq u\leq1$. If $q\leq-1$, then $u=-1$ and
\begin{equation*}
    g^+(q,u)=0,
    \qquad
    g^-(q,u)=-2(q+1)\geq0.
\end{equation*}
If $-1<q<1$, then $u=q$, so both $g^+(q,u)$ and $g^-(q,u)$ vanish. Finally, if $q\geq1$, then $u=1$ and
\begin{equation*}
    g^+(q,u)=2(q-1)\geq0,
    \qquad
    g^-(q,u)=0.
\end{equation*}
Hence $(q,u)\in\mathcal{S}$ in every case, and therefore
\begin{equation*}
    u=\varphi(q)
    \Longrightarrow
    (q,u)\in\mathcal{S}.
\end{equation*}
Thus, the equivalence
$
    u=\varphi(q)
    \Longleftrightarrow
    (q,u)\in\mathcal{S}
$
is proven.
Notice that $-1 \leq u \leq 1$ is equivalent to $(u+1)(1-u)\geq0$, while $v\in\mathcal V$ is equivalent to $(v-\ubar v)(\obar v-v)\geq0$. Extending this logic componentwise to $\R^m$ and substituting $q=Ke+v$ gives the result in~\eqref{eq:sat-semialg}.
\end{proof}
\begin{remark}
    This representation extends to arbitrary saturation bounds $\ubar{u} \leq u \leq \obar{u}$ by replacing $(u+ \pmb{1}_m)$, $(\pmb{1}_m-u)$ by $(u -\ubar{u})$, $(\obar{u}-u)$, respectively.
\end{remark}

Proposition~\ref{prop:semialg-sat} replaces the saturation nonlinearity by an equivalent semialgebraic description in the variables $(e, v, u)$. Accordingly, the equivalent saturated error process
\begin{equation}\label{eq:error-dynamics-semialg}
    e_{k+1} = Ae_k + B(u_k - v_k) + w_k, \quad (e_k, v_k, u_k) \in \mathcal{G}_K,
\end{equation}
will be used for the subsequent reachability analysis.

\section{Semialgebraic PRS Synthesis}\label{sec:set-synthesis}
The preceding section gives an exact semialgebraic representation of the error dynamics through the set $\mathcal{G}_K$. The analysis herein focuses on certifying PRS and PUB over $\mathcal{G}_K$, and hence for all admissible nominal input sequences $v_k\in\mathcal V$.
\subsection{Drift Condition}
The following result shows that PRS and PUB guarantees can be obtained by enforcing a stochastic Lyapunov drift condition over $\mathcal{G}_K$, adopting a contraction approach analogous to that present in~\cite{KaramPRS2025}.
\begin{theorem}\label{th:drift-prs-pub}
    Consider the nonlinear error dynamics in~\eqref{eq:sat-error-dynamics}. Suppose there exists a Lyapunov function $V: \R^n \to \R_{\geq 0}$ with $V(0) = 0$, $V(e) > 0$ for all $e \neq 0$, and $\lim_{\|e\| \to \infty} V(e) = \infty$, and constants $\lambda \in (0,1)$ and $\beta > 0$ such that
    \begin{equation}
        \hspace{-0.2em}
        \label{eq:drift-condition}
        \E_w \brackets{V\paren{Ae+B(u - v)+w}} \leq \lambda V(e) + \beta,
    \end{equation}
    for all $(e,v,u) \in \mathcal G_K$. Then, for $e_0=0$, the sublevel set
    \begin{equation}
        \label{eq:drift-prs}
        \mathcal{R}_k^\varepsilon = 
        \left\{
            e \in \R^n \mathrel{} \middle| \mathrel{} V(e) \leq \frac{1 - \lambda^k}{\varepsilon (1 - \lambda)} \beta
        \right\},
    \end{equation}
    is a $k$-step PRS with violation level $\varepsilon$ for the stochastic process ${\{ e_k \}}_{k \geq 0}$. Similarly, the sublevel set
    \begin{equation}
        \label{eq:drift-pub}
        \mathcal{R}^\varepsilon = 
        \left\{
            e \in \R^n \mathrel{} \middle| \mathrel{} V(e) \leq \frac{\beta}{\varepsilon (1 - \lambda)} + \zeta
        \right\},
    \end{equation}
    with $\zeta > 0$, is a PUB with violation level $\varepsilon$ for the stochastic process ${\{ e_k \}}_{k \geq 0}$.
\end{theorem}
\vspace{0.3em}
\begin{proof}
    For all $(e, v, u) \in \mathcal{G}_K$ defined by Proposition~\ref{prop:semialg-sat}, the error dynamics take the form $e_{k+1} = Ae_k + B(u_k - v_k) + w_k$. Thus,~\eqref{eq:drift-condition} gives
    \begin{equation*}
        \E_w\brackets{V(e_{k+1})} = \E\brackets{V(e_{k+1}) \mid e_k, v_k, u_k} \leq \lambda V(e_k) + \beta.
    \end{equation*}
    Taking expectations and applying the tower property gives
    \begin{equation*}
        \E\brackets{\E\brackets{V(e_{k+1}) \mid e_k, v_k, u_k}} = \E\brackets{V(e_{k+1})} \leq \lambda \E\brackets{V(e_k)} + \beta.
    \end{equation*}
    Since $V(0)=0$, iterating from $e_0 = 0$ yields
    \begin{equation}
        \label{eq:recursive-bound}
        \E\brackets{V(e_k)} \leq \sum_{i=0}^{k-1}\lambda^i \beta
        = \frac{1-\lambda^k}{1-\lambda}\beta,
    \end{equation}
    where the last equality uses the geometric series property.
    For $\beta > 0$, Markov's inequality applied to the nonnegative random variable $V(e_k)$, with
    $r_k = \beta \left( 1-\lambda^k \right)/(\varepsilon \cdot (1-\lambda))$, gives
    \begin{equation*}
        \Pr\{V(e_k) > r_k\} \leq \frac{\E\brackets{V(e_k)}}{r_k} \leq \varepsilon.
    \end{equation*}
    Equivalently, $\Pr\{e_k \in \mathcal{R}_k^\varepsilon\} \geq 1-\varepsilon$, proving the PRS claim. For the ultimate bound, the same drift recursion, now initialized at an arbitrary $e_0$, gives
    \begin{align*}
        \E \brackets{ V \left( e_k \right) } 
                    &\leq \lambda^k V \left( e_0 \right) + \frac{1 - \lambda^k}{1 - \lambda} \beta,\\
        \limsup_{k \to \infty}\E \brackets{ V \left( e_k \right) }
                    &\leq \frac{\beta}{1 - \lambda}.
    \end{align*}
    By the definition of limits, for every $\zeta > 0$ there exists a finite $t(e_0)$ such that $\E\brackets{V(e_k)} \leq \beta \cdot (1 - \lambda)^{-1} + \varepsilon \zeta$ for all $k \geq t(e_0)$.
    For $r = (\beta/\varepsilon) (1 - \lambda)^{-1} + \zeta$, Markov's inequality gives
    \begin{equation*}
        \Pr\{V(e_k) > r\} \leq \frac{\E\brackets{V(e_k)}} {r} \leq \varepsilon, \qquad \forall k \geq t(e_0).
    \end{equation*}
    Therefore, $\Pr\{e_k \in \mathcal{R}^{\varepsilon}\}\geq 1-\varepsilon$ for all $k \geq t(e_0)$.
\end{proof}

Theorem~\ref{th:drift-prs-pub} establishes a sufficient condition for the existence of valid PRS and PUB. The thresholds defining $\mathcal{R}_k^\varepsilon$ increase monotonically to $\beta / (\varepsilon (1 - \lambda))$, and consequently, $\mathcal{R}_k^\varepsilon \subseteq \mathcal{R}_{k+1}^\varepsilon \subseteq \mathcal{R}^\varepsilon$.

\begin{remark}
The use of Markov's inequality to obtain the PRS and PUB probability guarantees is undoubtedly conservative. Higher-degree polynomial $V$ can reduce conservatism through richer set geometry and higher-order moments, but does not remove that introduced by Markov's inequality. Reducing said conservatism would require additional assumptions on the noise or more elaborate moment-based approaches, which are deferred for future work.
\end{remark}

With sufficient conditions for the existence of valid PRS and PUB established, the next step is to optimize the geometry of these sets to reduce eventual conservatism in the nominal SMPC problem.

\subsection{PUB Geometry Optimization}\label{subsec:objective-func}
The nominal MPC uses the tightenings $\mathcal{Z}_k = \mathcal{X} \ominus \mathcal{R}_k^\varepsilon$ and $\mathcal{Z}_\infty =\mathcal{X} \ominus \mathcal{R}^\varepsilon$ for its stage and terminal constraints, respectively. Feasibility and performance considerations in SMPC therefore favor large tightened sets. Since $\mathcal{R}_k^\varepsilon \subseteq \mathcal{R}^\varepsilon$ for all $k$, the inclusion $\mathcal{Z}_\infty \subseteq \mathcal{Z}_k$ also holds for all $k$. Thus, enlarging $\mathcal{Z}_\infty$ expands a common subset of admissible states at all time steps. In line with this objective, the development herein focuses on optimizing the geometry of $\mathcal{R}^\varepsilon$.

The ideal design of $\mathcal{R}^\varepsilon$ would erode the state constraint set $\mathcal{X}$ only in directions that become active during closed-loop operation, thereby maximizing the nominal feasible region while guaranteeing the required probability of constraint satisfaction. In practice, however, the directions of active constraints are unknown a priori, as they depend on a variety of factors that only materialize during online MPC operation. Lacking any directional preference, a principled offline approach would be to maximize a geometry-preserving inner approximation of the actual tightening. To this end, consider the following results.

\begin{proposition}\label{prop:tightening-inner-approx}
Let $\mathcal{X} \subset \mathbb{R}^n$ be a convex set. If $\mathcal{R}^\varepsilon \subseteq \rho \mathcal{X}$ for some $\rho \in [0, 1]$, then
\begin{equation}\label{eq:tightening-inclusion}
    (1 - \rho) \mathcal{X} \subseteq \mathcal{X} \ominus \mathcal{R}^\varepsilon.
\end{equation}
\end{proposition}

\begin{proof}
    Let $z \in (1 - \rho) \mathcal{X}$ and $r \in \mathcal{R}^\varepsilon$. By definition, $z = (1 - \rho) x$ for some $x \in \mathcal{X}$. Likewise, the inclusion $\mathcal{R}^\varepsilon \subseteq \rho \mathcal{X}$ implies there exists some $y \in \mathcal{X}$ such that $r = \rho y$. Consequently, by convexity of $\mathcal{X}$,
    \begin{equation*}
        z + r = (1 - \rho)x + \rho y \in \mathcal{X}.
    \end{equation*}
    Since $z + r \in \mathcal{X}$ for every $r \in \mathcal{R}^\varepsilon$, it follows from the definition of the Pontryagin difference that $z \in \mathcal{X} \ominus \mathcal{R}^\varepsilon$. Since $z$ is arbitrary, the inclusion in~\eqref{eq:tightening-inclusion} holds.
\end{proof}

By imposing $\mathcal{R}^\varepsilon \subseteq \rho \mathcal{X}$, Proposition~\ref{prop:tightening-inner-approx} guarantees that the tightened set $\mathcal{Z}_\infty$ contains the homothetic inner approximation $(1-\rho)\mathcal{X}$. Geometry preservation, i.e., having $\mathcal{Z}_\infty$ approximate the geometry and size of $\mathcal{X}$, is then achieved by maximizing this inner approximation, which is equivalent to choosing the smallest feasible $\rho$. The following result characterizes the optimal value $\rho^\star$ and its relation to constraint tightening.
\begin{proposition}\label{prop:proxy-min}
Let $\mathcal{X}$ be a compact polyhedron with $0 \in \operatorname{int}(\mathcal{X})$. For any non-empty set $\mathcal{R}^{\varepsilon} \subseteq \R^n$, define
\begin{equation}\label{eq:min-uniform-scaling}
    \rho^\star
    =
    \inf\{\rho \geq 0 \mid \mathcal{R}^{\varepsilon} \subseteq \rho \mathcal{X}\}.
\end{equation}
Then
\begin{equation}\label{eq:min-uniform-scaling-support}
    \rho^\star
    =
    \max_{i=1,\dots,n_c}
    \frac{\delta_{\mathcal{R}^{\varepsilon}}(H_i)}{h_i},
\end{equation}
Equivalently, $\rho^\star$ is the infimum over all scalars $\rho \geq 0$ satisfying
\begin{equation}\label{eq:support-upper-bound}
    \delta_{\mathcal{R}^{\varepsilon}}(H_i) \leq \rho \, h_i,
    \qquad i=1,\dots,n_c.
\end{equation}
\end{proposition}

\begin{proof}
For any $\rho \geq 0$, $\rho\mathcal{X}$ is a polyhedron defined as $\rho \mathcal{X} = \{x \in \R^n \mid H_i^\top x \leq \rho \, h_i, \, \forall i =1, \dots, n_c\}$. Hence $\mathcal{R}^{\varepsilon} \subseteq \rho\mathcal{X}$ if and only if
\begin{equation*}
    H_i^\top e \leq \rho \, h_i, \qquad \forall e \in \mathcal{R}^{\varepsilon}, \quad \forall i=1,\dots,n_c.
\end{equation*}
Equivalently, taking the supremum over $e \in \mathcal{R}^{\varepsilon}$,
\begin{equation*}
    \delta_{\mathcal{R}^{\varepsilon}}(H_i) \leq \rho \, h_i, 
    \qquad \forall i=1,\dots,n_c .
\end{equation*}
Since $h_i>0$, the feasible scalars are those satisfying
\begin{equation*}
    \rho \geq
    \max_{i=1,\dots,n_c}
    \frac{\delta_{\mathcal{R}^{\varepsilon}}(H_i)}{h_i}.
\end{equation*}
Taking the infimum over this set gives~\eqref{eq:min-uniform-scaling-support}. The characterization~\eqref{eq:support-upper-bound} follows from the same inequalities.
\end{proof}
Proposition~\ref{prop:proxy-min} shows that minimizing $\rho$ minimizes the largest normalized support $\delta_{\mathcal R^\varepsilon}(H_i)/h_i$, thereby maximizing the guaranteed homothetic inner approximation $(1-\rho)\mathcal X$ of the tightened set. However, reductions in the non-worst-case supports do not improve the objective, and are therefore not favored, even though such reductions may enlarge the tightened set $\mathcal X\ominus\mathcal R^\varepsilon$. Consequently, PUBs with different geometries and sizes can yield the same objective value. A formulation using facet-dependent scaling factors $\rho_i$ could account for these improvements, but would require an additional criterion for weighing improvements across directions. A common $\rho$ provides a simple and interpretable objective, albeit slightly conservative, when no directional preference is available.

\begin{remark}
The scalar $\rho$ is used only as a surrogate objective for the synthesis of $\mathcal{R}^\varepsilon$; the actual constraint tightening remains $\mathcal{X}\ominus\mathcal{R}^\varepsilon$ and therefore retains the full geometry of the synthesized PUB.
\end{remark}


\subsection{Sum-of-Squares Certificates}\label{subsec:opt-prob}
Theorem~\ref{th:drift-prs-pub} and Propositions~\ref{prop:tightening-inner-approx} and~\ref{prop:proxy-min} establish sufficient criteria for synthesizing a geometrically optimized PUB. Because the dynamics in~\eqref{eq:error-dynamics-semialg} and the defining inequalities of $\mathcal G_K$ are polynomial in $(e,v,u)$, these conditions are amenable to sum-of-squares (SOS) optimization.

\begin{definition}[SOS Polynomial]\label{def:sos-poly}
    For $e \in \R^n$, a multivariate polynomial $p \in \R[e]$ of degree $2r$ is a sum-of-squares if $p(e) = \sum_{\ell = 1}^L q_\ell{(e)}^2$ for some finite collection of polynomials $q_\ell$ of degree at most $r$. Equivalently, $p(e) = \Phi_{0:r}{(e)}\T Q \, \Phi_{0:r}(e)$ for a Gram matrix $Q \psd$~\cite{parrilo2003semidefinite}.
\end{definition}

This Gram matrix equivalence allows SOS verification to be cast as a convex semidefinite program (SDP). Thus, the polynomial Lyapunov function can be parameterized as $V(e) = \Phi_{1:r}{(e)}^\top Q \, \Phi_{1:r}(e)$ with Gram matrix $Q \succeq 0$. Excluding the constant monomial directly enforces $V(0) = 0$. Define 
\begin{equation*}
    \Delta(e,v,u) = \lambda V(e) + \beta - \E_w\brackets{V \paren{Ae+B(u-v)+w}}.
\end{equation*}
The Lyapunov function must satisfy $\Delta(e, v, u) \geq 0$ for all $(e, v, u) \in \mathcal{G}_K$ and $V(e) \geq \gamma e^\top e$ for some $\gamma > 0$, to guarantee radial unboundedness and strict positivity for $e \neq 0$.
\begin{remark}
    Enforcing the drift requires an evaluation of the conditional expectation term in $\Delta(e, v, u)$. For compactness, let $F = Ae + B(u-v)$. Using multi-index notation, the multivariate binomial theorem gives $(F + w)^\alpha = \sum_{\nu \le \alpha} \binom{\alpha}{\nu} F^{\alpha - \nu} w^\nu$. Hence, the evaluated monomial basis can be factorized as $\Phi_{1:r}(F + w) = T(F)\Phi_{0:r}(w)$, where the matrix $T(F)$ collects the binomial coefficients and polynomials in $F$. Therefore, $\E_w\brackets{V \paren{F+w}} = \trace\paren{Q\,T(F)\,M_w\,T(F)\T}$, where $M_w = \E\brackets{\Phi_{0:r}(w)\Phi_{0:r}(w)^\top}$
    is the constant moment matrix containing disturbance moments up to degree $2r$.
\end{remark}

To enforce the drift over $\mathcal G_K$, let $g(e,v,u) \in \mathbb R^{4m}$ collect the scalar inequalities defining~\eqref{eq:sat-semialg}. Introducing the SOS multipliers $\sigma \in \Sigma[e,v,u]^{4m}$ yields the following sufficient generalized S-procedure certificate for nonnegativity of the drift over $\mathcal G_K$:
\begin{equation}\label{eq:domain-drift}
    p(e,v,u) = \Delta(e,v,u) - \sigma^\top g(e,v,u) \in \Sigma[e,v,u].
\end{equation}

Similarly, Propositions~\ref{prop:tightening-inner-approx} and~\ref{prop:proxy-min} require the PUB to be enclosed within $\rho \mathcal{X}$, such that $H_i^\top e \le \rho \, h_i$ for all $e$ satisfying $V(e) \leq \tau_\infty(\varepsilon)$, with
\begin{equation}
    \tau_\infty(\varepsilon) = \frac{\beta}{\varepsilon (1-\lambda) } + \zeta.     
\end{equation}
This inclusion is enforced for all facets $i = 1, \dots, n_c$, via SOS multipliers $\mu_i \in \Sigma[e]$ by
\begin{equation}
    \psi_i(e) = \rho\,h_i - H_i^\top e - \mu_i \left(\tau_\infty(\varepsilon) - V(e)\right) \in \Sigma[e].
\end{equation}

Collecting the objective and the multiplier certificates yields the following SOS optimization program
\begin{subequations}\label{eq:sos-program}
\begin{alignat}{2}
    \min_{\mathclap{\substack{V,\lambda,\rho,\\ \gamma,\sigma,\mu}}}\quad
        & \rho && \\
    \text{s.t.}\quad
        & \mathrlap{0 < \lambda < 1, \quad 0 < \rho \leq 1, \quad \gamma > 0,} && \\
        & V(e) - \gamma e^\top e \in \Sigma[e],\\
        & \sigma \in \Sigma[e,v,u]^{4m},\\
        & p(e,v,u) \in \Sigma[e,v,u],\\
        & \mu_i \in \Sigma[e],
          \quad && \forall i = 1, \dots, n_c,\\
        & \psi_i(e) \in \Sigma[e],
          \quad && \forall i = 1, \dots, n_c .
\end{alignat}
\end{subequations}
The decision variables are the function $V(e)$ (through the Gram matrix $Q$), the contraction factor $\lambda$, the scaling factor $\rho$, the positivity margin $\gamma$, and the SOS multipliers. The degrees of the Lyapunov function and multipliers, as well as $\tau_\infty$, are fixed a priori. 

Because the drift certificate features the product $\lambda Q$ and the containment certificate contains products between $\mu_i(e)$ and $V(e)$, the program~\eqref{eq:sos-program} is not jointly convex. A practical implementation therefore treats $\lambda$ as an outer scalar parameter and, for each fixed $\lambda$, solves a sequence of convex SOS sub-problems by alternating between $(Q,\sigma,\rho,\gamma)$ and $(\mu,\rho)$. For fixed $\lambda$, the alternating updates generate a non-increasing sequence of feasible objective values $\rho$, but global optimality is not guaranteed. The contraction factor is selected through a golden-section search~\cite{Kiefer1953minimax} over an appropriate interval. Since unimodality with respect to $\lambda$ is not established, the outer search likewise carries no global optimality guarantee. This process is laid out in Algorithm~\ref{alg:sos-synthesis}. Additionally, the feedback gain $K$ is treated as fixed throughout the synthesis.

\begin{algorithm}[h!]
\caption{Offline SOS Optimization}\label{alg:sos-synthesis}
\begin{algorithmic}[1]
\REQUIRE{Semialgebraic set $\mathcal{G}_K$, moment matrix $M_w$, violation level $\varepsilon$, fixed scalars $\tau_\infty>\zeta>0$, polynomial degree choices, search interval $[\lambda_{\min},\lambda_{\max}]\subset(0,1)$, tolerances $\eta$ and $\eta_\lambda$.}

\STATE{Set $\rho^\star \gets +\infty$.}

\item[]\textbf{Outer Search}
\FOR{each $\lambda$ queried by golden-section search on $[\lambda_{\min},\lambda_{\max}]$ to tolerance $\eta_\lambda$}
    \STATE{Set $\beta \gets \varepsilon(1-\lambda)(\tau_\infty-\zeta)$.}
    \STATE{Initialize $\mu_i \gets 1/\tau_\infty$ for all $i$, $\rho_{\mathrm{old}} \gets 1$, and $\mathtt{feasible}\gets\mathtt{true}$.}
    \FOR{$\ell = 0,1,\dots$}
            \STATE{\textbf{$\mathbf{Q}$-step:} With $\lambda$, $\beta$, and $\mu$ fixed, solve~\eqref{eq:sos-program} for $Q$, $\sigma$, $\rho$, $\gamma$.}
        \IF{infeasible}
            \STATE{Set $\mathtt{feasible}\gets\mathtt{false}$ and break.}
        \ENDIF{}
            \STATE{\textbf{$\boldsymbol{\mu}$-step:} With $\lambda$, $\beta$, $Q$, and $\sigma$ fixed, solve~\eqref{eq:sos-program} for $\mu$, $\rho$.}
        \IF{infeasible}
            \STATE{Set $\mathtt{feasible}\gets\mathtt{false}$ and break.}
        \ELSIF{$\rho_{\mathrm{old}} - \rho < \eta$}
            \STATE{break.}
        \ENDIF{}
        \STATE{Set $\rho_{\mathrm{old}} \gets \rho$.}
    \ENDFOR{}
    \STATE{Return $\rho$ to the outer search if $\mathtt{feasible}$, and $+\infty$ otherwise.}
    \IF{$\mathtt{feasible}$ and $\rho < \rho^\star$}
        \STATE{Store $(Q^\star,\lambda^\star,\rho^\star,\gamma^\star,\sigma^\star,\mu^\star)$.}
    \ENDIF{}
\ENDFOR{}
\IF{$\rho^\star=+\infty$}
    \RETURN{infeasible.}
\ENDIF{}
\RETURN{$V^\star(e)$, $\lambda^\star$, $\rho^\star$.}
\end{algorithmic}
\end{algorithm}

\begin{remark}\label{rem:fixed-tau}
The PUB level $\tau_\infty$ is fixed rather than optimized through $\beta$. For any fixed $\lambda$ selected by the outer search and any violation level $\varepsilon$, prescribing $\tau_\infty > \zeta$ uniquely determines the corresponding drift offset as $\beta = \varepsilon(1-\lambda)\left(\tau_\infty-\zeta\right)$. This entails no loss of generality for the SOS search, since a common positive rescaling of $V$, $\tau_\infty$, $\gamma$, and $\zeta$ preserves the drift, positivity, and containment certificates.
\end{remark}

\section{Numerical Simulations}\label{sec:numerical-sims}
The proposed method is evaluated on three illustrative examples, all of which share the same toy system dynamics, with $e \in \R^2$ and $u \in \R^1$, but differ in either the noise distribution or the saturation bounds. The focus of this evaluation is twofold: demonstrating the reduction in conservatism with respect to standard quadratic-LMI derived sets, and showcasing the geometric flexibility of higher-degree polynomial level sets, which is particularly evident when the disturbance is skewed or the saturation bounds are asymmetric. The code used in this section is available online\footnote{https://github.com/CarloKaram/sos-preach}. All simulations were carried out in Julia using JuMP~\cite{Lubin2023jump} and SumOfSquares~\cite{legat2017sos} in conjunction with the MOSEK solver~\cite{mosek}. The computations were executed on a Linux workstation equipped with an Intel i7-11850H CPU (16 cores) and 16GB of RAM. 

\subsection{Practical Considerations}
Because the system is subject to unbounded additive disturbances, the domain over which the drift is enforced is likewise unbounded. Relying purely on local multiplier polynomials to enforce the drift condition on an unbounded domain leaves the formulation vulnerable to numerical artifacts, where solver tolerances can cause the highest-degree terms of the polynomial to collapse in the deep tails. More broadly, SOS optimization in particular, and semidefinite programming in general, are known to suffer from a wide spectrum of numerical issues. It is therefore useful to perform at least a basic sanity check on SDP solver outputs, which the implementation of this work performs. The interested reader is referred to~\cite{Roux2017validating} for a deeper look into such problems and their potential systematic, formal solutions.

Scalability presents an additional practical limitation. Appendices~\ref{app:semialg-complexity} and~\ref{app:pwa-complexity} detail the dimensions of the main Gram matrices and the number of PSD blocks for the semialgebraic and PWA formulations, respectively. In the PWA formulation, the drift Gram matrices involve $n+m$ variables, but a separate certificate must be enforced over each of the $3^m$ saturation regions. The number of drift certificates and associated multiplier blocks therefore grows exponentially with $m$, compounding the inherent growth of dense SOS programs with system dimension and polynomial degree. The semialgebraic representation instead introduces the lifted input $u$, so its Gram matrices involve $n+2m$ variables and are somewhat larger. However, it requires only one global drift certificate and $4m$ multiplier blocks, yielding linear rather than exponential growth in the number of PSD blocks. This tradeoff makes the semialgebraic representation markedly more desirable as $m$ increases, while the overall method remains intended for moderately sized systems.

Symmetry may partly reduce this burden in the PWA formulation. In particular, when $\mathcal V=\mathcal U=[\ubar{u},\obar{u}]$, symmetry is inherent to the effective error-input mapping, even if the actuator bounds are asymmetric. Let $c=\ubar{u}+\obar{u}$. Componentwise saturation then satisfies $\varphi(c-q)=c-\varphi(q)$, while $v\in\mathcal U$ implies $c-v\in\mathcal U$. Taking $q=Ke+v$ therefore gives
\begin{equation*}
    \varphi(-Ke+c-v)-(c-v)
    =-\left(\varphi(Ke+v)-v\right).
\end{equation*}
Thus, every effective error input associated with $(e,v)$ is paired with its negative at $(-e,c-v)$. For a centrally symmetric disturbance distribution and $\mathcal V=\mathcal U$, restricting $V$ to be even therefore pairs the certificate over each region $\boldsymbol\Xi$ with that over the reflected region $-\boldsymbol\Xi$, while the fully unsaturated region is self-symmetric. It is then sufficient to enforce $(3^m+1)/2$ drift certificates rather than $3^m$. The semialgebraic formulation already uses one global drift certificate, so it admits no analogous reduction in the number of drift conditions. More generally, SOS polynomials invariant under the corresponding reflection admit Gram representations separated into even- and odd-degree monomials, reducing the dimensions of their individual Gram blocks. This decomposition can be applied to the PWA and semialgebraic formulations alike. Such reductions will be investigated for the semialgebraic representation in future work.

\subsection{Simulation Setup and Results}
The numerical example uses the stable system
\begin{equation*}
A = \begin{bmatrix} 0.89 & 0.10 \\ 0.10 & 0.89 \end{bmatrix}, \;
B = \begin{bmatrix} 0 \\ 1 \end{bmatrix}, \;
K = \begin{bmatrix} -0.282 & -0.8415 \end{bmatrix},
\end{equation*}
with state constraint set $\mathcal{X} = {[-10,10]}^2$. Here, the gain $K$ matches that of~\cite[Sec.~IV]{KaramPRS2025}, designed using the LMI constraints in~\cite[Rem.~2]{karam2026SMPC}. The examples use a violation level $\varepsilon=0.2$ and noise covariance $W=0.1 \cdot \mat{I}_2$. Consider the cases:
\begin{enumerate}
    \item[(a)] Gaussian disturbance $w \sim \mathcal{N}(0, W)$; symmetric saturation bounds $v,u \in [-1,1]$.
    \item[(b)] Shifted Gamma disturbance $w_i=\bar{w}_i-0.1$, with $\bar{w}_i\sim\Gamma(0.1,1)$, which yields $\E[w]=0$ and $\operatorname{Var}(w)=W$; symmetric saturation bounds $v,u\in[-1,1]$.
    \item[(c)] Gaussian disturbance $w \sim \mathcal{N}(0, W)$; asymmetric saturation bounds $u\in[-10,1]$ and $v\in[-5,1]$.
\end{enumerate}

Figure~\ref{fig:pub-comparison} compares the optimized PUBs for quartic ($2r=4$) and octic ($2r=8$) polynomials with those based on the quadratic LMI construction of~\cite[Sec.~III.B]{KaramPRS2025}. In all cases, the LMI-based construction yields $\rho\approx0.75$, while the polynomial certificates yield less conservative bounds and capture the geometry of the error distribution more accurately. For instance, in case~(a), the scaling factor decreases to $\rho\approx0.46$ for the degree-eight certificate. The higher-degree PUBs of case~(a) exhibit a somewhat ellipsoidal shape, but with flatter sides induced by the saturation nonlinearity. In cases~(b) and~(c), the minimal scaling similarly decreases to $\rho\approx0.59$ and $\rho\approx0.46$, respectively, when $2r=8$. Interestingly, while the octic PUB in case~(c) is visibly smaller than that of case~(a), both have the same $\rho$, an observation in line with the discussion in Section~\ref{subsec:objective-func}. The geometric advantage is also more pronounced in these asymmetric cases. The shifted Gamma disturbance induces a skewed error distribution, while the asymmetric input bounds modify the active saturation regimes. The polynomial level sets reflect these features without imposing an ellipsoidal geometry, as they exploit the higher-order structure of the resulting error distribution. Conversely, the quadratic LMI-based construction cannot distinguish cases~(a) and~(b) through the disturbance model due to the common covariance; in case~(c), it must enforce saturation limits $v, u\in[-1,1]$, the largest symmetric subset of $\mathcal U$, given its inability to handle asymmetric bounds.
\newpage

\begin{figure*}[t]
    \centering
    \includegraphics[width=0.85\textwidth]{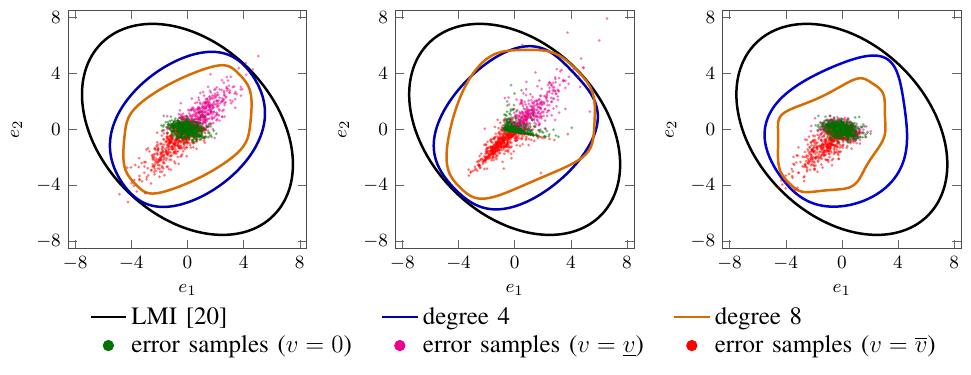}
    \caption{Synthesized PUBs for cases~(a),~(b), and~(c) (left-to-right). Red samples correspond to simulated error states at $k = 100$.}\label{fig:pub-comparison}
\end{figure*}

Figure~\ref{fig:asym-contours} more clearly shows how this geometry varies with the Lyapunov level in case~(c). For small levels, the relevant states remain close to the origin, where the dynamics are nearly linear, so the resulting sublevel sets (equivalently, PRS) are almost ellipsoidal. Larger levels reach regions where the saturation nonlinearity is active, and the asymmetric actuator bounds visibly deform the certified sublevel sets. This cannot be captured by a quadratic Lyapunov function, whose sublevel sets are homothetic ellipsoids: the $k$-step PRS and the PUB only differ by scale. Instead, a polynomial Lyapunov function generates a nested family of semialgebraic sublevel sets whose shape may vary, allowing the PRS and PUB to reflect different regimes of the perturbed saturated dynamics.

\begin{figure*}[h]
    \centering
    \includegraphics[width=0.85\textwidth]{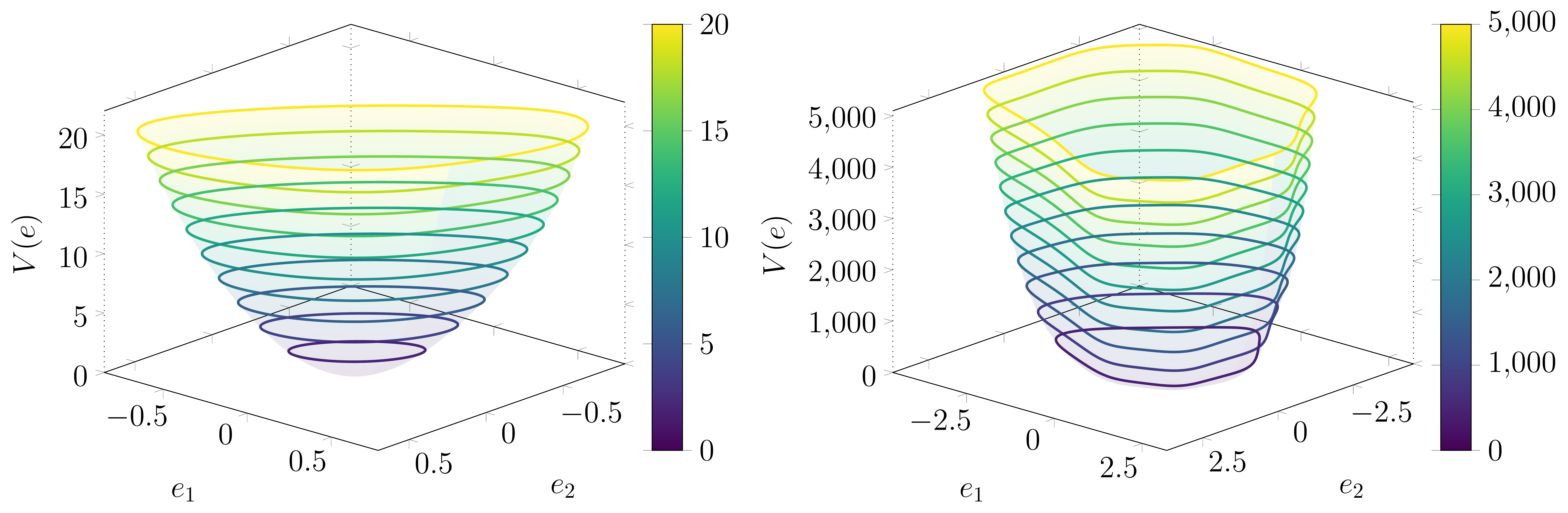}
    \caption{Selected sublevel contours (PRS) of the certified Lyapunov function $V(e)$ for case~(c) with $2r=8$, shown for $V(e)\leq 5000$. The function is positive definite. Lower levels remain close to ellipsoidal, while levels closer to the PUB threshold display the asymmetric geometry induced by the saturation bounds.}\label{fig:asym-contours}
\end{figure*}

The computational costs associated with the semialgebraic and PWA formulations are reported in Tables~\ref{tab:scalability-sa} and~\ref{tab:scalability-pwa}, respectively. Both are equivalent, exact representations of the saturation nonlinearity and yield the same synthesized sets; the comparison therefore concerns only the computational cost of enforcing the drift condition. For an octic $V(e)$ and $(n,m)=(2,1)$, $(2,2)$, and $(3,2)$, the semialgebraic formulation requires $11.52$~s, $54.38$~s, and $931.24$~s, whereas the base PWA formulation requires $31.54$~s, $512.54$~s, and $8439.96$~s. The semialgebraic formulation is thus approximately $2.7$, $9.4$, and $9.1$ times faster. Pairing symmetric PWA regions reduces the corresponding times to $21.77$~s, $237.97$~s, and $3829.02$~s, which remain approximately $1.9$, $4.4$, and $4.1$ times larger than for the semialgebraic formulation. The numbers of outer $\lambda$ evaluations and alternating iterations are identical in the two formulations for these tests, indicating that the difference is due to the underlying SDP structure. The advantage of the semialgebraic representation therefore becomes particularly clear once $m=2$. 
\newpage

\begin{table*}[t]
\centering
\caption{Computational cost of the semialgebraic formulation for increasing system dimensions and Lyapunov degrees under Gaussian noise.}
\label{tab:scalability-sa}
\setlength{\tabcolsep}{4pt}
\begin{tabular}{@{}c ccc ccc ccc@{}}
\toprule
& \multicolumn{3}{c}{$\deg(V)=2$}
& \multicolumn{3}{c}{$\deg(V)=4$}
& \multicolumn{3}{c}{$\deg(V)=8$} \\
\cmidrule(lr){2-4}\cmidrule(lr){5-7}\cmidrule(lr){8-10}
Dimensions $(n,m)$
& Iterations & $\lambda$ evals. & Solver time (s)
& Iterations & $\lambda$ evals. & Solver time (s)
& Iterations & $\lambda$ evals. & Solver time (s) \\
\midrule
$(2,1)$ & 2 & 8  & 5.47   & 3 & 9  & 6.27   & 3 & 10 & 11.52 \\
$(2,2)$ & 2 & 12 & 19.86  & 3 & 13 & 31.83  & 4 & 14 & 54.38 \\
$(3,2)$ & 2 & 13 & 129.96 & 3 & 14 & 353.72 & 4 & 15 & 931.24 \\
\bottomrule
\end{tabular}
\vspace{0.4em}
\begin{minipage}{0.96\textwidth}
\footnotesize
\vspace{0.5em}
\textit{Iterations} is the number of full SOS solves, with two SDP subproblems per iteration. The number of values evaluated by the golden-section search is reported under $\lambda$ \textit{evals.}, and \textit{solver time} is the total solver time over all iterations and $\lambda$ evaluations.
\end{minipage}
\end{table*}

\begin{table*}[t]
\centering
\caption{Computational cost of the PWA formulation for increasing system dimensions and Lyapunov degrees under Gaussian noise.}
\label{tab:scalability-pwa}
\setlength{\tabcolsep}{4pt}

\begin{tabular}{@{}c cccc cccc cccc@{}}
\toprule
&
\multicolumn{4}{c}{$\deg(V)=2$}
&
\multicolumn{4}{c}{$\deg(V)=4$}
&
\multicolumn{4}{c}{$\deg(V)=8$}
\\
\cmidrule(lr){2-5}
\cmidrule(lr){6-9}
\cmidrule(lr){10-13}

Dimensions $(n,m)$
&
Iterations & $\lambda$ evals. & Base (s) & Sym. (s)
&
Iterations & $\lambda$ evals. & Base (s) & Sym. (s)
&
Iterations & $\lambda$ evals. & Base (s) & Sym. (s)
\\
\midrule

$(2,1)$
&
$2$ & 8  & 35.76   & 6.91
&
$3$ & 9  & 12.52   & 11.03
&
$3$ & 10 & 31.54   & 21.77
\\

$(2,2)$
&
$2$ & 12 & 188.25  & 142.05
&
$3$ & 13 & 226.95  & 129.46
&
$4$ & 14 & 512.54  & 237.97
\\

$(3,2)$
&
$2$ & 13 & 1046.70 & 635.86
&
$3$ & 14 & 1689.20 & 963.25
&
$4$ & 15 & 8439.96 & 3829.02
\\
\bottomrule
\end{tabular}
\begin{minipage}{0.96\textwidth}
\footnotesize
\vspace{0.5em}
\textit{Iterations} is the number of full SOS solves, with two SDP subproblems per iteration, and $\lambda$ \textit{evals.} is the number of values evaluated by the golden-section search. ``Base'' and ``Sym.'' report the total solver time over all iterations and $\lambda$ evaluations for the base and symmetry-reduced formulations, respectively. ``Base'' enforces all $3^m$ PWA drift certificates, whereas ``Sym.'' restricts $V$ to be even and enforces one certificate for each pair of reflected regions, giving $(3^m+1)/2$ drift certificates.
\end{minipage}
\end{table*}

Fixing the feedback gain $K$ prior to the SOS synthesis introduces a further potential source of conservatism. In the semialgebraic formulation, $K$ changes the set $\mathcal G_K$ but enters its constraints affinely, while the lifted state update is expressed directly in terms of $(e,v,u)$ and does not contain $K$. Consequently, for fixed multipliers $\sigma$, the gain could theoretically be optimized alongside $(\mu,\rho)$. A complete controller--certificate co-design nevertheless requires a separate treatment, since changing $K$ trades stronger local contraction against a smaller unsaturated region. To assess the practical impact of fixing $K$, the synthesis was repeated for three stabilizing gains designed by discrete-time LQR:
\begin{equation*}
\begin{aligned}
\text{Weak:}\quad
&Q=0.1I,\quad R=10I, &K&=\begin{bmatrix}-0.05063759&-0.06655147\end{bmatrix},\\
\text{Moderate:}\quad
&Q=I,\qquad R=I,
&K&=\begin{bmatrix}-0.28087959&-0.56851480\end{bmatrix},\\
\text{Aggressive:}\quad
&Q=10I,\quad R=0.1I,
&K&=\begin{bmatrix}-0.41406979&-0.91700025\end{bmatrix}.
\end{aligned}
\end{equation*}
Table~\ref{tab:gain-sensitivity} reports the optimized $(\lambda,\rho)$ for the Gaussian, shifted-Gamma, and asymmetric cases.

\begin{table}[H]
\centering
\caption{Sensitivity of the optimized $(\lambda,\rho)$ to varying discrete-time LQR gains.}
\label{tab:gain-sensitivity}
\footnotesize
\setlength{\tabcolsep}{7pt}
\begin{tabular*}{\textwidth}{@{\extracolsep{\fill}}cc rr rr rr@{}}
\toprule
& & \multicolumn{2}{c}{$r=1$} & \multicolumn{2}{c}{$r=2$} & \multicolumn{2}{c}{$r=4$} \\
\cmidrule(lr){3-4}\cmidrule(lr){5-6}\cmidrule(lr){7-8}
Case & Gain
& $\lambda$ & $\rho$
& $\lambda$ & $\rho$
& $\lambda$ & $\rho$ \\
\midrule
& Weak       & 0.980 & 0.710 & 0.982 & 0.501 & 0.991 & 0.504 \\
Gaussian & Moderate   & 0.980 & 0.710 & 0.984 & 0.530 & 0.986 & 0.454 \\
& Aggressive & 0.980 & 0.710 & 0.984 & 0.544 & 0.986 & 0.457 \\
\midrule
\addlinespace[2pt]
& Weak       & 0.980 & 0.710 & 0.981 & 0.577 & 0.981 & 0.562 \\
Shifted Gamma & Moderate   & 0.980 & 0.710 & 0.981 & 0.583 & 0.982 & 0.565 \\
& Aggressive & 0.980 & 0.710 & 0.982 & 0.595 & 0.981 & 0.565 \\
\addlinespace[2pt]
\midrule
& Weak       & 0.980 & 0.710 & 0.982 & 0.500 & 0.984 & 0.437 \\
Asymmetric & Moderate   & 0.980 & 0.710 & 0.986 & 0.531 & 0.986 & 0.451 \\
& Aggressive & 0.980 & 0.710 & 0.987 & 0.558 & 0.986 & 0.453 \\
\bottomrule
\end{tabular*}
\end{table}

For $r=1$, the result is effectively insensitive to the tested gain. For $r=2$, the weak gain gives the smallest $\rho$ in all three cases. For $r=4$, the moderate gain is best in the Gaussian case, whereas the weak gain is best in the shifted-Gamma and asymmetric cases. Thus, no single gain is uniformly preferable across polynomial degrees and disturbance scenarios, and the improvements in $\rho$ are generally no greater than $8\%$ with respect to the base construction with $K = \begin{bmatrix} -0.282 & -0.8415 \end{bmatrix}$ for a given case and degree. Less aggressive gains enlarge the region of linearity but decrease the contraction within that region, whereas aggressive gains provide stronger linear contraction over a more restricted region of linearity. The scalar $\rho$ by itself does not accurately capture this tradeoff. In the PWA formulation, $K$ enters the region constraints non-affinely, even for fixed $V$. It must therefore be fixed a priori, with no clear way to integrate its optimization into the synthesis when $V$ has a high-dimensional SOS parameterization.

\section{Conclusion}\label{sec:conc}
This letter presented a semialgebraic framework to compute probabilistic reachable sets and ultimate bounds for saturated systems under unbounded stochastic disturbances. These sets are over-approximated by sublevel sets of an SOS Lyapunov function satisfying a drift criterion over a semialgebraic set representing the saturation. An SOS program optimizes this function, shaping its sublevel sets such that the resultant PUB is optimized for geometry-preserving constraint tightening. As numerically demonstrated, this significantly reduces the conservatism of the bounds. Future work will explore the moment-SOS hierarchy to directly optimize over probability and occupation measures, replacing current conservative Markov-based bounds with a sequence of converging semidefinite relaxations.

\appendices
\section{Regions of Saturation Model}\label{subsec:sat-regions}
For the development in this appendix, we redefine the admissible input sets using arbitrary componentwise bounds as
\begin{equation*}
    \mathcal{U} = \left\{u \in \R^m \mid \ubar{u} \leq u \leq \obar{u}\right\},
    \qquad
    \mathcal{V} = \left\{v \in \R^m \mid \ubar{v} \leq v \leq \obar{v}\right\}
    \subseteq \mathcal{U}.
\end{equation*}
Here, $\varphi$ denotes componentwise saturation onto $\mathcal U$. Define the deterministic part of the error dynamics as
\begin{equation}\label{eq:deterministic-error-map}
    f(e,v)=Ae+B\paren{\varphi(Ke+v)-v},
\end{equation}
so that $e_{k+1}=f(e_k,v_k)+w_k$.
An intuitive exact representation of the saturation is the so-called regions of saturation model, which exploits the piecewise affine (PWA) structure of $\varphi(\cdot)$. Following the approach of~\cite{daSilva1999}, the state-input space can be partitioned into distinct regions of saturation, where for each component $i$ of the unsaturated input $Ke+v$, the following cases arise,
\begin{equation}
    \label{eq:saturation-cases}
    \varphi_i(Ke + v) =
    \begin{cases}
        \obar{u}_i, & \text{if } K_i e + v_i \geq \obar{u}_i, \\
        K_i e + v_i, & \text{if } \ubar{u}_i \leq K_i e + v_i \leq \obar{u}_i, \\
        \ubar{u}_i, & \text{if } K_i e + v_i \leq \ubar{u}_i.
    \end{cases}
\end{equation}

For $Ke+v \in \mathbb{R}^m$, the three possibilities in~\eqref{eq:saturation-cases} yield $3^m$ saturation patterns. Each pattern is encoded by a vector $\vct{\Xi} = \vecc{\xi_1 & \cdots & \xi_m}^{\top} \in {\{-1, 0, +1 \}}^m$, whose $i$-th entry specifies whether actuator $i$ is saturated at its upper limit ($\xi_i = +1$), lower limit ($\xi_i = -1$), or operating in the linear region ($\xi_i = 0$). By a slight abuse of notation, let the subscript $j$ hereafter denote region membership rather than vector components, and index these pattern vectors by $j = 1, \dots, 3^m$. Let $R_j^{\mathrm{sat}}$ and $c_j^{\mathrm{sat}}$ stack the saturation-pattern inequalities from~\eqref{eq:saturation-cases}, and $\nu_j = \left| \left\{i:\left(\Xi_j\right)_i=0\right\}\right|$ denote the number of unsaturated components in pattern $\Xi_j$. Since each saturated component contributes one inequality and each unsaturated component two, $R_j^{\mathrm{sat}}\in\mathbb{R}^{(m+\nu_j)\times(n+m)}$ and $c_j^{\mathrm{sat}} \in \R^{m + \nu_j}$. Since the nominal input satisfies $v \in \mathcal{V} \subseteq \mathcal{U}$, define the augmented region data
\begin{equation}
    \label{eq:augmented-saturation-region-data}
    R_j =
    \begin{bmatrix*}[r]
        \multicolumn{2}{c}{R_j^{\mathrm{sat}}}\\
        \mat{0}_{m\times n} & \mat{I}_m\\
        \mat{0}_{m\times n} & -\mat{I}_m
    \end{bmatrix*},
    \qquad
    c_j =
    \begin{bmatrix}
        c_j^{\mathrm{sat}}\\
        \obar{v}\\
        -\ubar{v}
    \end{bmatrix}.
\end{equation}
Each $\boldsymbol{\Xi}_j$ then corresponds to an admissible polyhedral saturation region, compactly expressed as
\begin{equation}
    \label{eq:saturation-region}
    \mathcal{S}_j =
    \left\{ (e, v) \in \mathbb{R}^{n+m} \mathrel{} \middle| \mathrel{} R_j \vecc{e \\ v} \leq c_j \right\},
\end{equation}
where the last $2m$ rows encode $v \in \mathcal{V}$. Thus, $\mathcal{S}_j$ is the admissible part of the $j$-th saturation cell. The collection $\{\mathcal{S}_1, \dots, \mathcal{S}_{3^m}\}$ covers $\R^n \times \mathcal{V}$ and has mutually disjoint interiors.

Inside any given region $\mathcal{S}_j$, the saturated control signal is an exact affine function of $(e, v)$,
\begin{equation}
    \label{eq:affine-saturation}
    \varphi(Ke + v) =
    \diag(\vct{1}_m - \abs{\vct{\Xi}_j}) (Ke + v) + d \paren{\vct{\Xi}_j},
\end{equation}
where the absolute value is intended componentwise and the offset vector $d(\vct{\Xi}_j) \in \mathbb{R}^m$ is
\begin{equation}
    \label{eq:offset-vector}
    d_i(\boldsymbol{\Xi}_j) =
    \begin{cases}
        \obar{u}_i, & \text{if } \xi_{i} = 1, \\
        0, & \text{if } \xi_{i} = 0, \\
        \ubar{u}_i, & \text{if } \xi_{i} = -1.
    \end{cases}
\end{equation}

Letting $\mat{D}_j \in \mathbb{R}^{m \times m}$ be the diagonal indicator matrix $\mat{D}_j = \diag(\vct{1}_m - \abs{\vct{\Xi}_j})$, and substituting~\eqref{eq:affine-saturation} into the dynamics gives, for every $(e, v) \in \mathcal{S}_j$,
\begin{equation}
    \label{eq:region-dynamics}
    f_j(e, v) = \bar{A}_j e + \bar{B}_j v + B d(\vct{\Xi}_j),
\end{equation}
where the region-specific system matrices are given by
\begin{equation}
    \label{eq:region-matrices}
    \bar{A}_j = A + B \mat{D}_j K,
    \qquad
    \bar{B}_j = B(\mat{D}_j - \mat{I}_m).
\end{equation}

Thus, on the admissible domain $\R^n \times \mathcal{V}$, the dynamics $e_{k+1} = f(e_k, v_k) + w_k$ are PWA, with $f(e, v) = f_j(e, v)$ for all $(e, v) \in \mathcal{S}_j$, and all $j = 1, \dots, 3^m$.

\section{Complexity of the Semialgebraic Formulation}\label{app:semialg-complexity}
Let $\deg(V)=2r$, let the multipliers for the constraints defining $\mathcal G_K$ and the containment multipliers have degrees $2r_\sigma$ and $2r_\mu$, respectively, and define
\begin{equation*}
    r_p=\max\{r,r_\sigma+1\}.
\end{equation*}
For the block count below, consider the $4m$ polynomial inequalities defining $\mathcal G_K$ in~\eqref{eq:sat-semialg}. For a dense SOS representation, the Gram matrix defining $V$ has dimension
\begin{equation*}
    N_V=\binom{n+r}{n}-1,
\end{equation*}
where the constant monomial is excluded. The global drift certificate and each of its multipliers have Gram-matrix dimensions
\begin{equation*}
    N_p^{\mathrm{sa}}=\binom{n+2m+r_p}{n+2m},
    \qquad
    N_\sigma^{\mathrm{sa}}=\binom{n+2m+r_\sigma}{n+2m},
\end{equation*}
respectively. The coefficient-matching equality for the drift certificate involves at most
\begin{equation*}
    N_{c,p}^{\mathrm{sa}}=\binom{n+2m+2r_p}{n+2m}
\end{equation*}
distinct monomial coefficients. The containment multipliers and certificates have dimensions
\begin{equation*}
    N_\mu=\binom{n+r_\mu}{n},
    \qquad
    N_\psi=\binom{n+r+r_\mu}{n}.
\end{equation*}
The PSD blocks consist of the Gram representation of $V$, its positivity certificate, the global drift certificate, the $4m$ multipliers for the semialgebraic saturation constraints, and two containment blocks for each of the $n_c$ facets. Their total number is therefore
\begin{equation*}
    N_{\mathrm{PSD}}^{\mathrm{sa}}=3+4m+2n_c.
\end{equation*}
Thus, the dimensions of the drift Gram matrices grow with $n+2m$ and the selected polynomial degrees, while the number of blocks introduced by the saturation description grows linearly with $m$.

\section{Complexity of the PWA Formulation}\label{app:pwa-complexity}
Consider the same degree choices as in Appendix~\ref{app:semialg-complexity}. In the PWA formulation of Appendix~\ref{subsec:sat-regions}, each drift certificate and its multipliers depend on $(e,v)\in\R^{n+m}$. Their Gram-matrix dimensions are
\begin{equation*}
    N_p^{\mathrm{PWA}}=\binom{n+m+r_p}{n+m},
    \qquad
    N_\sigma^{\mathrm{PWA}}=\binom{n+m+r_\sigma}{n+m},
\end{equation*}
and the coefficient-matching equality for each drift certificate involves at most
\begin{equation*}
    N_{c,p}^{\mathrm{PWA}}=\binom{n+m+2r_p}{n+m}
\end{equation*}
distinct monomial coefficients. The dimensions $N_V$, $N_\mu$, and $N_\psi$ remain as given in Appendix~\ref{app:semialg-complexity}.

If $\nu_j$ components are unsaturated in the $j$th region of saturation, the corresponding multiplier has $3m+\nu_j$ entries. Since
\begin{equation*}
    \sum_{j=1}^{3^m}\nu_j=m3^{m-1},
    \qquad
    \sum_{j=1}^{3^m}(3m+\nu_j)=10m\,3^{m-1},
\end{equation*}
the formulation contains $3^m$ drift certificates and $10m\,3^{m-1}$ multiplier blocks associated with the regions of saturation. Including the Gram representation of $V$, its positivity certificate, and the containment terms, the complete problem contains
\begin{equation*}
    N_{\mathrm{PSD}}^{\mathrm{PWA}}
    =2+3^m+10m\,3^{m-1}+2n_c
\end{equation*}
PSD blocks. Hence, the PWA formulation has smaller individual drift and multiplier blocks, whose dimensions grow with $n+m$, but both the number of drift certificates and the number of associated multiplier blocks grow exponentially with $m$.

\bibliography{references}

@book{tarbouriechStability2011,
  title = {Stability and Stabilization of Linear Systems with Saturating Actuators},
  author = {Tarbouriech, Sophie and Garcia, Germain and Gomes Da Silva Jr., Jo{\~a}o Manoel and Queinnec, Isabelle},
  year = {2011},
  publisher = {Springer},
  isbn = {978-0-85729-940-6 978-0-85729-941-3},
}

@inproceedings{KaramPRS2025, 
  title={Probabilistic Reachable Set Estimation for Saturated Systems with Unbounded Additive Disturbances}, 
  DOI={10.1109/cdc57313.2025.11312343}, 
  booktitle={2025 IEEE 64th Conference on Decision and Control (CDC)}, 
  publisher={IEEE}, 
  author={Karam, Carlo and Tacchi-Bénard, Matteo and Fiacchini, Mirko}, 
  year={2025}, 
  pages={4547–4553}
}

@misc{karam2026SMPC,
  title={A Stochastic Tube-Based {MPC} Framework with Hard Input Constraints}, 
  author={Carlo Karam and Matteo Tacchi and Mirko Fiacchini},
  year={2026},
  eprint={2602.19867},
  archivePrefix={arXiv},
  primaryClass={eess.SY},
  url={https://arxiv.org/abs/2602.19867}, 
}

@article{parrilo2003semidefinite, 
  title={Semidefinite programming relaxations for semialgebraic problems}, 
  volume={96}, 
  number={2}, 
  journal={Mathematical Programming}, 
  publisher={Springer Science and Business Media LLC}, 
  author={Parrilo, Pablo A.}, 
  year={2003}, 
  pages={293–320} }

@article{lasserre2001global,
  author = {Lasserre, Jean B.},
  title = {Global Optimization with Polynomials and the Problem of Moments},
  journal = {SIAM Journal on Optimization},
  volume = {11},
  number = {3},
  pages = {796--817},
  year = {2001},
  doi = {10.1137/S1052623400366802}
}

@article{althoff2021set,
  author = {Althoff, Matthias and Frehse, Goran and Girard, Antoine},
  title = {Set Propagation Techniques for Reachability Analysis},
  journal = {Annual Review of Control, Robotics, and Autonomous Systems},
  volume = {4},
  pages = {369--395},
  year = {2021},
  doi = {10.1146/annurev-control-071420-081941}
}

@article{kurzhanskiy2007ellipsoidal,
  author={Kurzhanskiy, Alex A. and Varaiya, Pravin},
  journal={IEEE Transactions on Automatic Control}, 
  title={Ellipsoidal Techniques for Reachability Analysis of Discrete-Time Linear Systems}, 
  year={2007},
  volume={52},
  number={1},
  pages={26-38},
  doi={10.1109/TAC.2006.887900}}

@inproceedings{girard2005reachability,
  author = {Girard, Antoine},
  title = {Reachability of Uncertain Linear Systems Using Zonotopes},
  booktitle = {Hybrid Systems: Computation and Control},
  series = {Lecture Notes in Computer Science},
  volume = {3414},
  pages = {291--305},
  publisher = {Springer},
  year = {2005},
  doi = {10.1007/978-3-540-31954-2_19}
}

@inproceedings{leguernic2009reachability,
  author = {Le Guernic, Colas and Girard, Antoine},
  title = {Reachability Analysis of Hybrid Systems Using Support Functions},
  booktitle = {Computer Aided Verification},
  series = {Lecture Notes in Computer Science},
  volume = {5643},
  pages = {569--581},
  publisher = {Springer},
  year = {2009}
}

@book{blanchini2015set,
  author = {Blanchini, Franco and Miani, Stefano},
  title = {Set-Theoretic Methods in Control},
  edition = {2},
  publisher = {Birkh{\"a}user},
  year = {2015}
}

@article{abate2008probabilistic,
  author = {Abate, Alessandro and Prandini, Maria and Lygeros, John and Sastry, Shankar},
  title = {Probabilistic Reachability and Safety for Controlled Discrete Time Stochastic Hybrid Systems},
  journal = {Automatica},
  volume = {44},
  number = {11},
  pages = {2724--2734},
  year = {2008},
  doi = {10.1016/j.automatica.2008.03.027}
}

@article{prajna2007framework,
  author = {Prajna, Stephen and Jadbabaie, Ali and Pappas, George J.},
  title = {A Framework for Worst-Case and Stochastic Safety Verification Using Barrier Certificates},
  journal = {IEEE Transactions on Automatic Control},
  volume = {52},
  number = {8},
  pages = {1415--1428},
  year = {2007},
  doi = {10.1109/TAC.2007.902736}
}

@inproceedings{anand2022kinductive,
  author = {Anand, Mahathi and Murali, Vishnu and Trivedi, Ashutosh and Zamani, Majid},
  title = {$k$-Inductive Barrier Certificates for Stochastic Systems},
  booktitle = {Proceedings of the 25th ACM International Conference on Hybrid Systems: Computation and Control},
  series = {HSCC '22},
  pages = {12:1--12:11},
  publisher = {ACM},
  year = {2022},
  doi = {10.1145/3501710.3519532}
}

@article{mesbah2016stochastic,
  author = {Mesbah, Ali},
  title = {Stochastic Model Predictive Control: An Overview and Perspectives for Future Research},
  journal = {IEEE Control Systems Magazine},
  volume = {36},
  number = {6},
  pages = {30--44},
  year = {2016},
  doi = {10.1109/MCS.2016.2602087}
}

@article{cannon2011stochastic,
  author = {Cannon, Mark and Kouvaritakis, Basil and Rakovi{\'c}, Sa{\v{s}}a V. and Cheng, Qifeng},
  title = {Stochastic Tubes in Model Predictive Control with Probabilistic Constraints},
  journal = {IEEE Transactions on Automatic Control},
  volume = {56},
  number = {1},
  pages = {194--200},
  year = {2011},
  doi = {10.1109/TAC.2010.2086553}
}

@article{fiacchini2021probabilistic,
  author = {Fiacchini, Mirko and Alamo, Teodoro},
  title = {Probabilistic Reachable and Invariant Sets for Linear Systems with Correlated Disturbance},
  journal = {Automatica},
  volume = {132},
  pages = {109808},
  year = {2021},
  doi = {10.1016/j.automatica.2021.109808}
}

@inproceedings{hewing2018stochastic,
  author = {Hewing, Lukas and Zeilinger, Melanie N.},
  title = {Stochastic Model Predictive Control for Linear Systems Using Probabilistic Reachable Sets},
  booktitle = {2018 IEEE Conference on Decision and Control (CDC)},
  pages = {5182--5188},
  publisher = {IEEE},
  year = {2018},
  doi = {10.1109/CDC.2018.8619554}
}

@article{magron2019discrete,
author = {Magron, Victor and Garoche, Pierre-Loic and Henrion, Didier and Thirioux, Xavier},
title = {Semidefinite Approximations of Reachable Sets for Discrete-time Polynomial Systems},
journal = {SIAM Journal on Control and Optimization},
volume = {57},
number = {4},
pages = {2799-2820},
year = {2019},
doi = {10.1137/17M1121044},
}

@misc{kordabad2025almost,
  title={Sum-of-Squares Certificates for Almost-Sure Reachability of Stochastic Polynomial Systems}, 
  author={Arash Bahari Kordabad and Rupak Majumdar and Sadegh Soudjani},
  year={2025},
  eprint={2510.25513},
  archivePrefix={arXiv},
  primaryClass={math.OC},
  url={https://arxiv.org/abs/2510.25513}, 
}

@article{Roux2017validating, 
  title={Validating numerical semidefinite programming solvers for polynomial invariants}, 
  volume={53}, 
  DOI={10.1007/s10703-017-0302-y}, 
  number={2}, 
  journal={Formal Methods in System Design}, 
  publisher={Springer Science and Business Media LLC}, 
  author={Roux, Pierre and Voronin, Yuen-Lam and Sankaranarayanan, Sriram}, 
  year={2017}, 
  month=Oct, 
  pages={286–312}
}

@article{daSilva1999, 
  title={Polyhedral regions of local stability for linear discrete-time systems with saturating controls}, 
  volume={44}, 
  ISSN={0018-9286}, 
  number={11}, 
  journal={IEEE Transactions on Automatic Control}, 
  publisher={Institute of Electrical and Electronics Engineers (IEEE)}, 
  author={Gomes Da Silva Jr., Jo{\~a}o Manoel and Tarbouriech, S.}, 
  year={1999}, 
  pages={2081–2085} 
 }

@ARTICLE{korda2021spatio,
  author={Cibulka, Vít and Korda, Milan and Haniš, Tomáš},
  journal={IEEE Control Systems Letters}, 
  title={Spatio-Temporal Decomposition of Sum-of-Squares Programs for the Region of Attraction and Reachability}, 
  year={2022},
  volume={6},
  number={},
  pages={812-817},
}

@article{fiacchini2026measured,
  author = {Fiacchini, Mirko and Mammarella, Martina and Dabbene, Fabrizio},
  title = {Measured-State Conditioned Recursive Feasibility for Stochastic Model Predictive Control},
  journal = {International Journal of Robust and Nonlinear Control},
  volume = {36},
  number = {11},
  pages = {5964-5981},
  year = {2026}
}

@article{Xue2026barrier,
author = {Bai Xue},
title = {Sufficient and necessary barrier-like conditions for safety and reach-avoid verification of stochastic discrete-time systems},
journal = {Automatica},
volume = {187},
pages = {112919},
year = {2026},
}

@article{Xue2025finite,
author = {Bai Xue},
title = {Finite-time safety and reach-avoid verification of stochastic discrete-time systems},
journal = {Information and Computation},
volume = {307},
pages = {105368},
year = {2025},
}

@article{Lubin2023jump,
    author = {Miles Lubin and Oscar Dowson and Joaquim {Dias Garcia} and Joey Huchette and Beno{\^i}t Legat and Juan Pablo Vielma},
    title = {{JuMP} 1.0: {R}ecent improvements to a modeling language for mathematical optimization},
    journal = {Mathematical Programming Computation},
    year = {2023},
    doi = {10.1007/s12532-023-00239-3}
}

@Conference{legat2017sos,
  author    = {Legat, Beno{\^\i}t and Coey, Chris and Deits, Robin and Huchette, Joey and Perry, Amelia},
  title     = {{Sum-of-squares optimization in Julia}},
  booktitle = {The First Annual JuMP-dev Workshop},
  year      = {2017},
}

@manual{mosek,
   author = "MOSEK ApS",
   title = "MOSEK Optimizer API for Julia 11.2.4.",
   year = 2025,
   url = "MOSEK Optimizer API for Julia 11.2.4"
 }

@article{korda2022stability,
  title={Stability and performance verification of dynamical systems controlled by neural networks: algorithms and complexity},
  author={Korda, Milan},
  journal={IEEE Control Systems Letters},
  volume={6},
  pages={3265--3270},
  year={2022},
  publisher={IEEE}
}

@article{Kiefer1953minimax,
 author = {Kiefer, Jack},
 journal = {Proceedings of the American Mathematical Society},
 number = {3},
 pages = {502--506},
 publisher = {American Mathematical Society},
 title = {Sequential Minimax Search for a Maximum},
 urldate = {2026-09-25},
 volume = {4},
 year = {1953}
}

\end{document}